\documentclass[11pt,a4paper]{article}
\usepackage{epsf,epsfig,amsfonts,amsgen,amsmath,amstext,amsbsy,amsopn,amsthm}
\usepackage{amsmath}
\usepackage{enumerate}
\usepackage{hhline}
\usepackage{multirow}
\usepackage{multicol}
\usepackage{booktabs}
\usepackage{lineno}
\usepackage{amsfonts,amsthm,amssymb,bm}
\usepackage{amsfonts}
\usepackage{enumitem}
\usepackage{graphics}
\usepackage{latexsym,bm}
\usepackage{amsfonts,amsthm,amssymb,bbding}
\usepackage{indentfirst}
\usepackage{graphicx}
\usepackage{authblk}

\usepackage{color}
\usepackage[colorlinks=true,anchorcolor=blue,filecolor=blue,linkcolor=red,urlcolor=blue,citecolor=blue]{hyperref}
\usepackage{float}
\usepackage{subcaption}
\usepackage{tikz,enumerate}
\usetikzlibrary{calc}
\usepackage{geometry}
\newcommand{\n}{\noindent}

\expandafter\let\expandafter\oldproof\csname\string\proof\endcsname
\let\oldendproof\endproof
\renewenvironment{proof}[1][\proofname]{%
	\oldproof[\normalfont\bfseries #1]%
}{\oldendproof}
\newenvironment{subproof}[1][\normalfont\it Subproof]{%
	\begin{proof}[#1]%
	}{%
	\end{proof}%
}

\newtheorem{theorem}{Theorem}[section]

\newtheorem{lemma}[theorem]{Lemma}
\newtheorem{claim}{Claim}[section]

\newtheorem{conjecture}[theorem]{Conjecture}

\renewcommand{\leq}{\leqslant}
\renewcommand{\le}{\leqslant}
\renewcommand{\geq}{\geqslant}
\renewcommand{\ge}{\geqslant}

\usepackage[numbers,sort&compress]{natbib}

\begin{document}

\title{The optimal $\chi$-bound for $\{P_6, \text{dart}, K_4\}$-free graphs}

\author[1]{Yidong Zhou\thanks{Email: 1120240329@mail.nankai.edu.cn}}
\author[2]{Kaiyang Lan\thanks{Email: kylan95@mnnu.edu.cn. Corresponding author.}}

\affil[1]{College of Computer Science, Nankai University, Tianjin, 300350, China}
\affil[2]{School of Mathematics and Statistics, Minnan Normal University, Zhangzhou, 363000, Fujian, China}

\date{\today}
\maketitle



\begin{abstract}
	A \textit{diamond} is a graph obtained from \(K_4\) by removing an edge, and a \textit{dart} 
	is a graph obtained from a diamond by adding a pendant edge to a vertex of degree 3.
	We prove that every $\{P_6, \text{dart}, K_4\}$-free graph is 6-colorable. This improves the previous bound of 7 due to Hong and Xu \cite{HongXu2025} and resolves their open question on the optimality of the bound.
	Our result also extends a theorem of Karthick and Mishra~\cite{KarthickMishra2018}, who proved 6-colorability for the class of \(\{P_6, \text{diamond}, K_4\}\)-free graphs.	
\end{abstract}



\section{Introduction}

All graphs in this note are finite, simple, and undirected.
For a positive integer $k$, a \textit{$k$-coloring} of a graph $G$ assigns a color from $\{1,2,\ldots,k\}$ to each vertex of $G$ such that adjacent vertices receive distinct colors.
We say that $G$ is \textit{$k$-colorable} if $G$ admits a $k$-coloring.
The \textit{chromatic number} of $G$, denoted by $\chi(G)$, is the minimum integer $k$ such that $G$ is $k$-colorable.
The \textit{clique number} of $G$, denoted by $\omega(G)$, is the maximum size of a clique in $G$.
A graph $G$ is \textit{perfect} if $\chi(H)=\omega(H)$ for every induced subgraph $H$ of $G$.

Let $P_t$, $C_t$, and $K_t$ denote the path, cycle, and complete graph on $t$ vertices, respectively.
For a set $\mathcal{H}$ of graphs, we say that $G$ is $\mathcal{H}$-free if $G$ contains no induced subgraph isomorphic to any member of $\mathcal{H}$.
A graph class is \textit{hereditary} if it can be defined by a set of forbidden induced subgraphs.
A hereditary graph class $\mathcal{G}$ is \textit{$\chi$-bounded}~\cite{Gyarfas1973} if there exists a function $f:\mathbb{N}\to\mathbb{N}$ such that for every $G\in\mathcal{G}$ and every induced subgraph $H$ of $G$, we have $\chi(H)\le f(\omega(H))$.
Such a function $f$ is called a \emph{$\chi$-binding function} for $\mathcal{G}$.
The class $\mathcal{G}$ is \textit{polynomially $\chi$-bounded} if $f$ can be chosen to be bounded above by a polynomial in $\omega(H)$; equivalently, there exists a polynomial $p$ such that $\chi(H)\le p(\omega(H))$ for every $G\in\mathcal{G}$ and every induced subgraph $H$ of $G$.
It is clear that $\chi(G)\ge \omega(G)$ for every graph $G$, but in general $\chi(G)$ cannot be bounded from above by any function of $\omega(G)$; for instance, there exist triangle-free graphs with arbitrarily large chromatic number~\cite{Erdos1959,Mycielski1955,Zykov1949}.

It is easy to see that a necessary condition for the class of $H$-free graphs to be $\chi$-bounded is that $H$ is a forest.
The famous Gy\'arf\'as--Sumner conjecture \cite{Gyarfas1973,Sumner1981} asserts the converse:

\begin{conjecture}[\cite{Gyarfas1973,Sumner1981}]
	For every forest $T$, the class of $T$-free graphs is $\chi$-bounded.
\end{conjecture}

Gy\'arf\'as \cite{Gyarfas1987} proved the conjecture for $T=P_t$: every $P_t$-free graph $G$ satisfies $\chi(G)\le (t-1)^{\omega(G)-1}$.
Note that this $\chi$-binding function is exponential in $\omega(G)$.
Esperet \cite{Esperet2017} conjectured that every $\chi$-bounded class is polynomially $\chi$-bounded.
This conjecture was recently disproved by Bria\'nski, Davies, and Walczak \cite{BrianskiDaviesWalczak2024}, who constructed $\chi$-bounded classes that are not polynomially $\chi$-bounded.
However, the question remains wide open for $P_t$-free graphs.

The answer is trivial for $t\le 4$, as $P_4$-free graphs are perfect.
For $t=5$, a recent result of Nguyen \cite{Nguyen2025} gives a polynomial bound.
Apart from these, little is known.
Consequently, researchers have turned their attention to subclasses of $P_t$-free graphs.
One natural direction is to study $P_t$-free graphs with additional forbidden induced subgraphs.
For instance, every $\{P_5,C_3\}$-free graph is 3-colorable, and the bound is tight as witnessed by $C_5$ \cite{WoegingerSgall2001}; every $\{P_5,C_4\}$-free graph satisfies $\chi(G)\le \lceil \frac{5}{4}\omega(G)\rceil$, and equal-size blowups of $C_5$ show optimality \cite{ChoudumKarthickShalu2007}.
For $t=6$, optimal $\chi$-binding functions have also been obtained for $C_3$ and $C_4$ \cite{BrandstadtKlembtMahfud2006,KarthickMaffray2019}.
For $t\ge 7$, Gravier, Ho\`ang, and Maffray \cite{GravierHoangMaffray2003} proved that every $\{P_t,C_3\}$-free graph $G$ satisfies $\chi(G)\le t-2$; whether this bound is optimal remains open.
More recently, Huang \cite{Huang2024} showed that every $\{P_7,C_4,C_5\}$-free graph $G$ satisfies $\chi(G)\le \frac{11}{9}\omega(G)$.

\begin{figure}[htbp]
	
	\centering
	
	\begin{subfigure}{.25\textwidth}
		\centering
		\tikzstyle{v}=[circle, draw, solid, fill=black, inner sep=0pt, minimum width=3pt]
		\begin{tikzpicture}
			\node[v, label=above:$u_1$] (u1) at (0, 1.2) {};
			\node[v, label=left:$u_2$] (u2) at (-1.2, 0) {};
			\node[v, label=right:$u_4$] (u3) at (1.2, 0) {};
			\node[v, label=below:$u_3$] (u4) at (0, -1.2) {};
			
			\draw[blue, thick] (u1) -- (u2);
			\draw[blue, thick] (u2) -- (u4);
			\draw[blue, thick] (u4) -- (u3);
			\draw[blue, thick] (u4) -- (u1);
			\draw[blue, thick] (u1) -- (u3);
		\end{tikzpicture}
		\caption{Diamond.}
		\label{fig:diamond}
	\end{subfigure}
	\begin{subfigure}{.25\textwidth}
		\centering
		\tikzstyle{v}=[circle, draw, solid, fill=black, inner sep=0pt, minimum width=3pt]
		\begin{tikzpicture}
			\node[v] (u1) at (0, 1.2) {};
			\node[v] (u2) at (-1.2, 0) {};
			\node[v] (u3) at (1.2, 0) {};
			\node[v] (u4) at (0, -1.2) {};
			\node[v] (u5) at (0, 2.4) {};
			
			\draw[blue, thick] (u1) -- (u2);
			\draw[blue, thick] (u2) -- (u4);
			\draw[blue, thick] (u4) -- (u3);
			\draw[blue, thick] (u4) -- (u1);
			\draw[blue, thick] (u1) -- (u3);
			\draw[blue, thick] (u1) -- (u5);
		\end{tikzpicture}
		\caption{Dart.}
		\label{fig:dart}
	\end{subfigure}
	\begin{subfigure}{.25\textwidth}
		\centering
		\tikzstyle{v}=[circle, draw, solid, fill=black, inner sep=0pt, minimum width=3pt]
		\begin{tikzpicture}
			\node[v] (u1) at (0, 1.2) {};
			\node[v] (u2) at (-1.2, 0) {};
			\node[v] (u3) at (1.2, 0) {};
			\node[v] (u4) at (0, -1.2) {};
			\node[v] (u5) at (0, 2.4) {};
			
			\draw[blue, thick] (u1) -- (u2);
			\draw[blue, thick] (u4) -- (u3);
			\draw[blue, thick] (u1) -- (u3);
			\draw[blue, thick] (u1) -- (u5);
		\end{tikzpicture}
		\caption{Chair.}
		\label{fig:chair}
	\end{subfigure}
	
	\caption{The graphs diamond, dart, and chair.}\label{fig:threegraphs}
\end{figure}
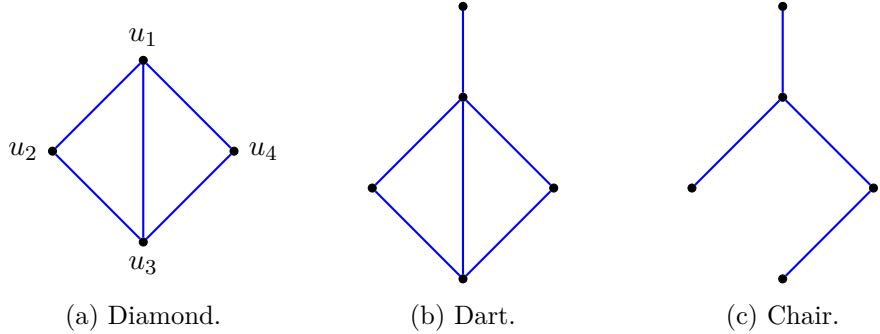

A \textit{diamond} is the graph obtained from $K_4$ by removing an edge, a \textit{dart} is the graph obtained from a diamond by adding a pendant edge to a vertex of degree 3 of the diamond.
A \textit{chair} is the graph obtained from $P_3$ by adding a pendant edge to each vertex (see \autoref{fig:threegraphs}).  
Karthick and Maffray \cite{KarthickMaffray2016} proved that $\chi(G) \le \omega(G) + 1$ for $\{P_5, \text{diamond}\}$-free graphs. Brause and Gei{\ss}er \cite{BrauseGeisser2021} presented a binding function $\Theta(\frac{\omega^2(G)}{\log \omega(G)})$ for $\{P_5, \text{dart}\}$-free graphs.
For related subclasses of \(P_6\)-free graphs, Karthick and Maffray \cite{KarthickMaffray2016} also proved that $\chi(G) \le \omega(G) + 1$ for every $\{P_6, \text{chair}, \text{diamond}\}$-free graph.
Cameron, Huang and Merkel~\cite{CameronHuangMerkel2021} proved that \(\chi(G) \le \omega(G) + 3\) for $\{P_6, \text{diamond}\}$-free graph.
 Lan, Zhou, and Liu \cite{LanZhouLiu2023} proved that \(\chi(G) \le \max\{\omega(G),3\}\) for \(\{P_6,C_4,\text{diamond}\}\)-free graphs.
The currently best known upper bound for $P_6$-free graphs is due to Gravier, Ho\`ang, and Maffray \cite{GravierHoangMaffray2003}.

\begin{theorem}[\cite{GravierHoangMaffray2003}]
	\label{thm:Gravier}
	Let $G$ be a $P_k$-free graph with $k \ge 4$ and $\omega(G) \ge 2$. Then \(\chi(G) \le (k-2)^{\omega(G)-1}\).
\end{theorem}


This means that every $\{P_6, K_3\}$-free graph is 4-colorable, and every $\{P_6, K_4\}$-free graph is 16-colorable.
A \textit{hole} is an induced cycle of length at least \(4\), and an \textit{odd hole} is a hole of odd length.
Chudnovsky, Seymour, Robertson, and Thomas \cite{ChudnovskySeymourRobertsonThomas2010} showed that every $\{\text{odd hole}, K_4\}$-free graph is 4-colorable.
In particular, every $\{P_6, C_5, K_4\}$-free graph is 4-colorable.
In \cite{KarthickMishra2018}, Karthick and Mishra proved that $\chi(G) \le 2\omega(G) + 5$ if $G$ is $\{P_6, \text{diamond}\}$-free, and $\chi(G) \le 6$ if $G$ is additionally $K_4$-free.
Cameron, Huang, and Merkel \cite{CameronHuangMerkel2021} improved these results and answered a question raised in \cite{KarthickMishra2018}.

Very recently, Hong and Xu \cite{HongXu2025} proved the following theorem.

\begin{theorem}[\cite{HongXu2025}]\label{thm:hongxu}
	Every $\{P_6, \text{dart}, K_4\}$-free graph is 7-colorable.
\end{theorem}

They asked whether the bound can be improved to 6, and noted that the optimal bound, if lower, would be 5. 
Their computer search showed that no 6-critical graph of order at most 21 exists. 
We prove that 6 is in fact the optimal bound, thereby completely resolving their question.

\begin{theorem}\label{thm:main}
	Every $\{P_6, \text{dart}, K_4\}$-free graph is 6-colorable. Moreover, this bound is tight.
\end{theorem}
	
	We note that Theorem~\ref{thm:main} also strengthens a result of Karthick and Mishra \cite{KarthickMishra2018}, who proved that every \(\{P_6, \text{diamond}, K_4\}\)-free graph is 6-colorable, since the latter class is properly contained in the class of \(\{P_6, \text{dart}, K_4\}\)-free graphs.
	
	The main technical contribution is contained in Lemmas~\ref{lem:bowtie}--\ref{lem:quotient-bipartite} and Proposition~\ref{prop:W4}.
	
	The rest of the paper is organized as follows.
	In Section~\ref{part2}, we establish the structural tools needed for the proof, including the auxiliary graph construction that is the main novelty of our approach.
	In Section~\ref{part3}, we prove Theorem~\ref{thm:main}.
	
	We close this section with some notation and terminology.
For a graph \(G\), let \(\delta(G)\) denote its minimum degree.
For two disjoint vertex sets \(X,Y\subseteq V(G)\), we say that \(X\) is \textit{complete} to \(Y\) if every vertex of \(X\) is adjacent to every vertex of \(Y\), and \textit{anti-complete} to \(Y\) if no vertex of \(X\) has a neighbor in \(Y\).
For any \(S \subseteq V(G)\), \(G[S]\) denotes the subgraph of \(G\) induced by \(S\), and \(G\setminus S\) denotes the subgraph of \(G\) induced by \(V(G)\setminus S\).
A vertex subset \(K \subseteq V(G)\) is a \textit{clique cutset} if \(K\) is a clique and \(G\setminus K\) has more components than \(G\).
For a vertex \(x \in V(G)\) and a set \(S \subseteq V(G)\), let \(N_S(x)\) denote the set of neighbors of \(x\) in \(S\), i.e., \(N_S(x)=N(x)\cap S\).
For \(v \in V(G)\), \(N_G(v)\) denotes the neighborhood of \(v\), and \(d_G(v)=|N_G(v)|\) is the degree of \(v\) in \(G\).
When the graph \(G\) is clear from the context, we omit the subscript.

\section{Tools}\label{part2}

In this section, we collect the necessary tools for the proof of Theorem~\ref{thm:main}. 
We use two known results.
The first one is due to Cameron, Huang and Merkel:

\begin{theorem}[\cite{CameronHuangMerkel2021}]\label{thm:Cameron}
	Let $G$ be a $\{P_6, \text{diamond}\}$-free graph. Then \(\chi(G) \le \omega(G) + 3\).
\end{theorem}

Second, we shall use the structural reductions from Hong and Xu \cite{HongXu2025}, 
which were used to prove that every $\{P_6,\text{dart},K_4\}$-free graph is 7-colorable.
For completeness, we state precisely the parts needed below.

Let $G$ be a connected $\{P_6,\text{dart},K_4\}$-free graph with no clique cutset, and let $H$ be an induced diamond with vertex set
\[
V(H)=\{u_1,u_2,u_3,u_4\}
\]
and edge set
\[
E(H)=\{u_1u_2,u_2u_3,u_3u_4,u_4u_1,u_1u_3\}.
\]
Thus $u_2u_4\notin E(G)$, and $u_1,u_3$ are the two degree-three vertices of the diamond (see \autoref{fig:diamond}).

For $x\in V(G)\setminus V(H)$, partition vertices according to their neighborhoods in $H$ as follows:
\begin{align*}
	A_i&=\{x:N_H(x)=\{u_i\}\},\quad i\in\{1,2,3,4\},\\
	B_{ij}&=\{x:N_H(x)=\{u_i,u_j\}\},
\end{align*}
where $ij\in\{12,23,34,41,13,24\}$,
\begin{align*}
	C_i&=\{x:N_H(x)=\{u_i,u_{i+1},u_{i+2}\}\},\quad i\in\{1,2,3,4\},
\end{align*}
with indices taken modulo $4$,
\[
D=\{x:N_H(x)=V(H)\},\qquad Z=\{x:N_H(x)=\emptyset\}.
\]


The following lemma is a key result of \cite{HongXu2025}. 
We give a new vision of proof for this lemma to ensure completeness of our result. 

\begin{lemma}[\cite{HongXu2025}]\label{lem:3-colorable for Z}
    Let $G$ and $H$ be as above. 
    Then $\chi(Z)\leq 3$. 
\end{lemma}
\begin{proof}
    Let $Z_0$ be a component of $Z$. 
    The following claim gives the structure of $Z_0$ 
    \begin{claim}\label{clm:structure of Z}
        Let $r$ be a vertex in $N(H)$ that has neighbors in $Z_0$. 
        Then 
        \begin{itemize}
            \item[$(1)$] every vertex in $N(r)\cap Z_0$ is either complete or anti-complete to each component of $Z_0\setminus N(r)$; 
            \item[$(2)$] each component of $N(r)\cap Z_0$ and $Z_0\setminus N(r)$ is either a $K_2$ or $K_1$. 
        \end{itemize}
    \end{claim}
    \begin{subproof}[Proof of Claim \ref{clm:structure of Z}]
        Suppose first that $r$ has a neighbor $v\in Z_0$ that has a neighbor and a non-neighbor in a component of $Z_0\setminus N(r)$. 
        So there exist $x,y\in Z_0\setminus N(r)$ with $rx,xy\in E(G)$ but $ry\notin E(G)$. 
        Since $D=\emptyset$, there is an index $i\in[4]$ such that $r$ is adjacent to $u_i$ but is not adjacent to $u_{i+1}$. 
        Then $u_{i-1}-u_i-r-v-x-y$ is an induced $P_6$, a contradiction. 
        This proves Claim \ref{clm:structure of Z} (1). 

        Note that since $Z_0$ is connected and by Claim \ref{clm:structure of Z} (1), for every component $L$ of $N(r)\cap Z_0$ and $Z_0\setminus N(r)$, there is a vertex $r'$ that is complete to $L$ ($r'=r$ if $L\subseteq N(r)\cap Z_0$). 
        Suppose that there is an induced $P_3=a-b-c$ in $L$. 
        Then $\{a,b,c,r'\}$ induces a diamond and hence, $\{a,b,c,r,r'\}$ induces a dart if $r'\neq r$ or $\{a,b,c,r,w\}$ induces a dart if $r'= r$, where $w\in V(H)$ is a neighbor of $r$. 
        So $L$ is $P_3$-free, which implies that $L$ is a clique. 
        Since $G$ has no $K_4$, $L$ is either a $K_2$ or a $K_1$. 
        This proves Claim \ref{clm:structure of Z} (2).
    \end{subproof}
    Since $G$ has no clique cutset, there exist two nonadjacent vertices $b,c\in N(H)$ such that $b,c$ have neighbors in $Z_0$. 
    Let $F=N(b)\cap Z_0$ ,$S=F\setminus N(c)$ and $T=F\cap N(c)$. 
    
    \begin{claim}\label{clm:if and only if of Q}
        Let $Q=xy$ be an edge in $Z_0\setminus N(b)$. 
        Then 
        \begin{itemize}
            \item[$(1)$] $N(Q)\cap F\subseteq S$ if $Q\cap N(c)\neq \emptyset$; 
            \item[$(2)$] $N(Q)\cap F\subseteq T$ if $Q\cap N(c)= \emptyset$. 
        \end{itemize}
    \end{claim}
    \begin{subproof}[Proof of Claim \ref{clm:if and only if of Q}]
        Suppose first that $Q\cap N(c)\neq \emptyset$ and there is a vertex $v\in T$ that is adjacent to $Q$.  
        By symmetry and Claim \ref{clm:structure of Z} (1) with $r=b$, we may assume that $xc\in E(G)$ and $v$ is complete to $Q$. 
        Since $G$ has no $K_4$, $cy\notin E(G)$ and hence $\{b,c,v,x,y\}$ induces a dart, a contradiction. 
        This proves Claim \ref{clm:if and only if of Q} (1). 

        We then suppose that $Q\cap N(c)= \emptyset$ and there is a vertex $v\in S$ that is adjacent to $Q$.  
        By Claim \ref{clm:structure of Z} (1) with $r=b$, $v$ is complete to $Q$. 
        So $\{v,x,y\}$ induces a $K_3$ in $Z_0\setminus N(c)$, which contradicts Claim \ref{clm:structure of Z} (2) with $r=c$. 
        This proves Claim \ref{clm:if and only if of Q} (2). 
    \end{subproof}
    \begin{claim}\label{clm:an edge anti to Q}
        Let $aa'$ be an edge of $F$. 
        Then $Q$ is anti-complete to $\{a,a'\}$ if $a,a'\in S$ or $a,a'\in T$. 
    \end{claim}
    \begin{subproof}[Proof of Claim \ref{clm:an edge anti to Q}]
        Suppose to contrary that $Q$ is not anti-complete to $\{a,a'\}$. 
        Since $G$ has no $K_4$ and by Claim \ref{clm:structure of Z} (1) with $r=b$, we may assume that $a$ is complete to $Q$ but $a'$ is anti-complete to $Q$. 
        Suppose first that $a,a'\in S$. 
        By Claim \ref{clm:if and only if of Q}, $Q\cap N(c)\neq \emptyset$, say $xc\in E(G)$, which contradicts Claim \ref{clm:structure of Z} (1) with $r=c$ since $x$ is adjacent to $a$ but is nonadjacent to $a'$. 
        So we may assume that $a,a'\in T$. 
        By Claim \ref{clm:if and only if of Q} again, $Q\cap N(c)= \emptyset$. 
        Then $\{a,a',b,c,x\}$ induces a dart, a contradiction. 
        This proves Claim \ref{clm:an edge anti to Q}.
    \end{subproof}
    We are now ready to give a 3-coloring of $Z_0$. 
    \begin{itemize}
        \item Color all isolated vertices of $S$ with $1$; color all isolated vertices of $T$ with $2$; color each $K_2$ of $S$ or $T$ with $1,2$; \item color all isolated vertices of $Z_0\setminus F$ with $3$; color each $K_2$ of $Z_0\setminus F$ with $2,3$ if $c$ has neighbors in this $K_2$; color each $K_2$ of $Z_0\setminus F$ with $1,3$ if $c$ has no neighbors in this $K_2$. 
    \end{itemize}
    By Claims \ref{clm:structure of Z}-\ref{clm:an edge anti to Q}, this is a 3-coloring of $Z_0$. 
    This completes the proof of Lemma \ref{lem:3-colorable for Z}.
\end{proof}

We shall use the following facts from \cite{HongXu2025}.
\begin{lemma}[\cite{HongXu2025}]\label{lem:HX}
	Let $G$ and $H$ be as above.  
    Then the following hold.
    \begin{itemize}
        \item[$(1)$] $A_1=A_3=C_1=C_3=D=\emptyset$.
        \item[$(2)$] Each component of $G[A_2]$, $G[A_4]$ and $G[B_{24}]$ is a clique of size at most $2$.
        \item[$(3)$] $B_{13}$ is a stable set, and $|B_{i(i+1)}|\leq 1$ for each $i\in\{1,2,3,4\}$.
        \item[$(4)$] Either $B_{13}=\emptyset$, or $B_{12}=B_{23}=B_{34}=B_{41}=\emptyset$.
        \item[$(5)$] If $C_2\neq\emptyset$ or $C_4\neq\emptyset$, then $\chi(G)\leq 6$ (contained in Case 1 of \cite{HongXu2025} with $\chi(Z)\leq 3$ by Lemma \ref{lem:3-colorable for Z}). 
    \end{itemize}
\end{lemma}

In the case $C_2=C_4=\emptyset$ and $B_{13}\neq\emptyset$, we also need the following part of Case 2.1 of \cite{HongXu2025}.

\begin{lemma}[Case 2.1 of \cite{HongXu2025}]\label{lem:HXres}
	Assume that \(C_2=C_4=\emptyset, B_{13}\neq\emptyset\), and hence, by Lemma \ref{lem:HX} (4), $B_{12}=B_{23}=B_{34}=B_{41}=\emptyset$.
	Then $V(G)=V(H)\cup A_2\cup A_4\cup B_{13}\cup B_{24}\cup Z$.
	Moreover, 
    \begin{itemize}
        \item[$(1)$] $\chi(G[A_4\cup B_{24}])\leq 2$.
        \item[$(2)$] Let $Z_1$ be the union of the components of $G[Z]$ that have a neighbor in $A_2$, and let $Z_2=Z\setminus Z_1$.  Then $Z_2$ is anti-complete to $Z_1\cup A_2\cup V(H)$, and $\chi(G[Z_2])\leq 3$ (by Lemma \ref{lem:3-colorable for Z}).
        \item[$(3)$] If $Q$ is a component of $G[Z_1]$, then $G[Q]$ is $K_1$ or $K_2$, and $N_{A_2}(Q)$ is complete to $Q$.
    \end{itemize}
\end{lemma}

\medskip

We shall assume the hypotheses of Lemma~\ref{lem:HXres} throughout the rest of this section.
Put
$$W:=V(H)\cup A_2\cup B_{13}\cup Z.$$
We prove that $G[W]$ is $4$-colorable. 

We shall use colors \(1,2,3,4\). Our intended coloring gives color \(1\) to \(u_1\), color \(3\) to \(u_3\), color \(2\) to both \(u_2\) and \(u_4\), and also color \(2\) to every vertex of \(B_{13}\).
This is compatible with $H\cup B_{13}$ because $u_2u_4\notin E(G)$, because $B_{13}$ is stable, and because every vertex of $B_{13}$ is adjacent in $H$ exactly to $u_1$ and $u_3$.
Therefore vertices of $A_2$ must be colored only with colors $1$ and $3$, since each vertex of $A_2$ is adjacent to $u_2$.  By Lemma~\ref{lem:HX} (2), $G[A_2]$ is a matching together with isolated vertices.  The difficulty is to coordinate this $2$-coloring of $A_2$ with the $K_2$-components of $G[Z_1]$.

We say a component $Q=\{q,q'\}$ of $G[Z_1]$ \textit{bad} if $qq'\in E(G)$ and both $q$ and $q'$ have a neighbor in $B_{13}$.  If $Q$ is not bad, then at least one of its two vertices has no neighbor in $B_{13}$, and that vertex may later receive color $2$ while the other receives color $4$.  If $Q$ is bad, then neither endpoint of $Q$ may receive color $2$, and therefore it will be convenient to require $N_{A_2}(Q)$ to be monochromatic in the colors $1,3$.

We now define an auxiliary graph $S$ on vertex set $A_2$ as follows.
For distinct $x,y\in A_2$, put $xy\in E(S)$ if and only if there exists a bad component $Q$ of $G[Z_1]$ such that \(x,y\in N_{A_2}(Q)\).
Thus each component of $S$ is intended to be monochromatic.

The key point is that no edge of $G[A_2]$ has both ends in the same component of $S$, and that after contracting the components of $S$, the matching edges of $G[A_2]$ form a bipartite graph.

\begin{lemma}\label{lem:bowtie}
	There exist no vertices $x,y,z\in A_2$ and bad components \(Q=\{q,q'\},\quad R=\{r,r'\}\) of $G[Z_1]$ such that \(xy\in E(G[A_2])\), $Q$ is complete to $\{x,z\}$, and $R$ is complete to $\{z,y\}$.
\end{lemma}

\begin{proof}
	Assume that such $x,y,z,Q,R$ exist. 
    Since $G$ has no $K_4$, $Q$ is complete to $\{x,z\}$ and $R$ is complete to $\{z,y\}$, $xz,yz\notin E(G)$ and hence $z\notin \{x,y\}$. 
	Also $Q\neq R$, for otherwise the same component is complete to $\{x,y\}$ which gives a $K_4$.
	
	Since $Q$ is bad, choose $b\in B_{13}$ adjacent to $q$. 
	We first prove that $bq'\in E(G)$. 
    Suppose to the contrary that $bq'\notin E(G)$.
	If $bx,bz\notin E(G)$, then $\{b,q,q',x,y\}$ induces a dart, a contradiction. 
	Hence $b$ is adjacent to at least one of $x,z$.
	If $bx\in E(G)$, then $\{u_2,b,q,q',x\}$ induces a dart, a contradiction. 
	The same argument with $z$ in place of $x$ gives a contradiction if $bz\in E(G)$.
	Thus $bq'\in E(G)$.
	
	Now $b$ is adjacent to both $q$ and $q'$.  
    Since $G$ has no $K_4$, $bx,bz\notin E(G)$. 
    Since $u_4-u_1-b-q-x-y$ is not an induced $P_6$, $by\in E(G)$. 
	Let $s\in\{r,r'\}$. 
    Since $u_4-u_1-b-q-z-s$ is not an induced $P_6$, $bs\in E(G)$. 
    So $b$ is complete to $R$ and hence $\{b,y,r,r'\}$ induces a $K_4$, a contradiction.
    
	This proves Lemma~\ref{lem:bowtie}.
\end{proof}

\begin{lemma}\label{lem:no-loop}
	Let $x,y\in A_2$. 
    If $xy\in E(G)$, then $x$ and $y$ lie in distinct components of $S$.
\end{lemma}
\begin{proof}
	Suppose not.
	Let $x=z_0,z_1,\ldots,z_t=y$ be a shortest path from $x$ to $y$ in $S$. 
    Since $G$ has no $K_4$, $t\geq 2$. 
	If $t=2$, then, by the definition of $S$, there are bad components $Q$ and $R$ such that $Q$ is complete to $\{z_0,z_1\}=\{x,z_1\}$ and $R$ is complete to $\{z_1,z_2\}=\{z_1,y\}$.
	This contradicts Lemma~\ref{lem:bowtie} since $xy\in E(G)$.
	Hence $t\geq 3$.
	
	Let $Q=qq'$ be a bad component witnessing the edge $z_0z_1$, and let $R=rr'$ be a bad component witnessing the edge $z_1z_2$. 
	By the minimality of the path, $Q\neq R$; otherwise $Q$ is complete to $z_0$ and $z_2$, giving an edge $z_0z_2$ and a shorter path between $x$ and $y$ in $S$. 
    
    Consider path $y-x-q-z_1-r-z_2$. 
    If $y$ is adjacent to $q$, then $y$ is complete to $Q$ by Lemma \ref{lem:HXres} (3), and so $\{q,q',x,y\}$ induces a $K_4$, a contradiction. 
    So we may assume that $qy\notin E(G)$. 
    Since each component of $A_2$ is either a $K_2$ or a $K_1$, $yz_1,yz_2,xz_2\notin E(G)$. 
    Since $Q,R$ are distinct bad component, $xr,qz_2\notin E(G)$. 
    Then $yr\in E(G)$; otherwise $y-x-q-z_1-r-z_2$ is an induced $P_6$. 
    This implies that $yz_1$ is an edge in $S$ and hence $x-z_1-y$ is an induced $P_3$ in $S$, which contradicts $t\geq 3$. 
	This proves Lemma~\ref{lem:no-loop}.
\end{proof}
Let $\Gamma$ be the graph obtained by contracting each component of $S$ to a single vertex and adding an edge between two components whenever there is an edge of $G[A_2]$ with one end in each component. 
\begin{lemma}\label{lem:quotient-bipartite}
	 $\Gamma$ is bipartite.
\end{lemma}

\begin{proof}
	By Lemma~\ref{lem:no-loop}, no edge of $G[A_2]$ has both ends in the same component of $S$, so $\Gamma$ has no loop.
	Suppose that $\Gamma$ contains an odd cycle, and choose a shortest odd cycle $C=C_0C_1\cdots C_{k-1}C_0$, where $k\geq 3$ is odd.
	For each $i$ modulo $k$, choose an edge \(x_i y_i\in E(G[A_2])\) with $x_i\in C_i$ and $y_i\in C_{i+1}$.
	Inside the component $C_{i+1}$, let $P_i$ be a shortest path from $y_i$ to $x_{i+1}$ in $S$.
	
	We first show that every $P_i$ has length one.
	Suppose some $P_i$ has length at least $2$, and write its first three vertices as \(z_0=y_i, z_1, z_2\).
	Let $Q$ be a bad component witnessing the edge $z_0z_1$, and let $R$ be a bad component witnessing the edge $z_1z_2$.
	Choose $q\in Q$ and $r\in R$.
    Since $z_0,z_1,z_2\in C_i$, $\{z_0,z_1,z_2\}$ is a stable set. 
    By the minimality of $P_i$, $q$ is not adjacent to $z_2$, $r$ is not adjacent to $y_i$. 
    Since $x_i\notin C_{i+1}$, $x_iz_1,x_iz_2\notin E(G)$ by Lemma~\ref{lem:no-loop}. 
    Since each component of $A_2$ is either a $K_2$ or a $K_1$, $x_iq,x_ir\notin E(G)$. 
    It follows that $x_i-y_i-q-z_1-r-z_2$ is an induced $P_6$, a contradiction. 
	Therefore every $P_i$ has length one.
	
	Consequently, for each $i$ there exists a bad component $Q_i$ such that $Q_i$ is complete to $\{y_i,x_{i+1}\}$.
	Take the first two consecutive edges $x_0y_0,x_1y_1$ of the cycle, and choose vertices \(q_0\in Q_0, q_1\in Q_1\).
	Consider the path $x_0-y_0-q_0-x_1-y_1-q_1$. 
    Since each component of $A_2$ is either a $K_2$ or a $K_1$, $x_0x_1,x_0y_1,y_0x_1,y_0y_1\notin E(G)$. 
    By Lemma \ref{lem:HXres} (3) and since $G$ has no $K_4$, $x_0q_0,x_1q_1,y_1q_0\notin E(G)$. 
    By Lemma \ref{lem:no-loop}, $x_0q_1,y_0q_1\notin E(G)$. 
    It follows that $x_0-y_0-q_0-x_1-y_1-q_1$ is an induced $P_6$, a contradiction. 
	
	Hence $\Gamma$ has no odd cycle, and so $\Gamma$ is bipartite.
	This proves Lemma~\ref{lem:quotient-bipartite}.
\end{proof}

\begin{lemma}\label{prop:W4}
    Under the hypotheses of Lemma~\ref{lem:HXres}, the graph $G[W]$ is $4$-colorable.
\end{lemma}

\begin{proof}
	By Lemma~\ref{lem:quotient-bipartite}, the graph $\Gamma$ is bipartite.
	Color the two sides of $\Gamma$ with colors $1$ and $3$, and give every vertex of $A_2$ the color of its component in $S$.
	Then every component of $S$ is monochromatic, and every edge of $G[A_2]$ has its ends in opposite sides of $\Gamma$, hence receives colors $1$ and $3$.
	Now we give color \(1\) to \(u_1\), color \(3\) to \(u_3\), color \(2\) to both \(u_2\) and \(u_4\), and also color \(2\) to every vertex of \(B_{13}\).
	This is proper on $H\cup B_{13}$, and no vertex of $A_2$ receives color $2$.
	
	It remains to color $Z$.  Let $Q$ be a component of $G[Z_1]$.  By Lemma~\ref{lem:HXres} (3), $Q$ is $K_1$ or $K_2$.	
	If $Q$ is $K_1$, color its unique vertex $4$.
	If $Q=\{q,q'\}$ is $K_2$ and is not bad, choose an endpoint, say $q$, with no neighbor in $B_{13}$, and gives \(2\) to \(q\), and color \(4\) to \(q'\).
	This is proper since $Q\subseteq Z$, all neighbors of $Q$ in $A_2$ use colors $1,3$, and the vertex colored $2$ has no neighbor in $B_{13}$.
	If $Q=\{q,q'\}$ is bad, then all vertices in $N_{A_2}(Q)$ lie in one component of $S$, and hence have the same color, either $1$ or $3$.
	If that color is $1$, color $q,q'$ with $3$ and $4$; if that color is $3$, color $q,q'$ with $1$ and $4$.
	This is proper since $Q$ is anti-complete to $H$, both endpoints avoid color $2$, and they avoid the common color of their neighbors in $A_2$.
	
	Finally, by Lemma~\ref{lem:HXres} (2), $Z_2$ is anti-complete to $Z_1\cup A_2\cup V(H)$ and $\chi(G[Z_2])\leq 3$.  
    Color $G[Z_2]$ using colors $1,3,4$.  
    This causes no conflict with $B_{13}$, since every vertex of $B_{13}$ has color $2$, and it causes no conflict with $Z_1\cup A_2\cup V(H)$ because $Z_2$ is anti-complete to that set.
	
	Thus $G[W]$ is $4$-colorable.
	This proves Proposition~\ref{prop:W4}.
\end{proof}


	\section{Proof of the main theorem}\label{part3}
	
	\begin{proof}[\bf Proof of Theorem~\ref{thm:main}]
		Suppose the theorem is false, and let $G$ be a counterexample with $|V(G)|$ minimum.  
        Then $G$ is connected.  
        Moreover, if $K$ is a clique cutset, each block together with $K$ is a proper induced subgraph and hence is $6$-colorable; since $|K|\leq 3$, the color names can be permuted so that the colorings agree on $K$, giving a $6$-coloring of $G$. 
        So we may assume that $G$ has no clique cutset. 
		
		By minimality, every proper induced subgraph of $G$ is $6$-colorable.
		Hence $\chi(G)=7$. 
		Thus $G$ is $7$-critical, and in particular $\delta(G)\geq 6$. 
		If $G$ has no induced diamond, then $G$ is $\{P_6,\text{diamond}\}$-free.
		By the Theorem~\ref{thm:Cameron}, $\chi(G)\leq \omega(G)+3$. 
		Since $G$ is $K_4$-free, $\omega(G)\leq 3$, and therefore $\chi(G)\leq 6$, a contradiction.
		Hence $G$ contains an induced diamond $H$ as in Section~\ref{part2}.
		Use the corresponding partition of $V(G)\setminus V(H)$ as in Section~\ref{part2}.
		By Lemma~\ref{lem:HX} (5), if $C_2\neq\emptyset$ or $C_4\neq\emptyset$, then $\chi(G)\leq 6$, a contradiction.
		Therefore $C_2=C_4=\emptyset$.

		We claim that $B_{13}\neq\emptyset$.  
        Suppose $B_{13}=\emptyset$.
		Then, by Lemma~\ref{lem:HX} (1) and since $C_2=C_4=\emptyset$, the only vertices outside $H$ adjacent to $u_1$ can lie in $B_{12}\cup B_{41}$.
		By Lemma~\ref{lem:HX} (3), $|B_{12}|\leq 1, |B_{41}|\leq 1$.
		Thus $d_G(u_1)\leq 3+1+1=5$, contradicting $\delta(G)\geq 6$.
		Hence $B_{13}\neq\emptyset$.
		By Lemma~\ref{lem:HX} (4), $B_{12}=B_{23}=B_{34}=B_{41}=\emptyset$.
		Together with Lemma~\ref{lem:HX} (1) and $C_2=C_4=\emptyset$, this gives $V(G)=V(H)\cup A_2\cup A_4\cup B_{13}\cup B_{24}\cup Z$.
		Set $W=V(H)\cup A_2\cup B_{13}\cup Z$ and $U=A_4\cup B_{24}$.
		By Lemma~\ref{prop:W4}, $\chi(G[W])\leq 4$.
		By Lemma~\ref{lem:HXres} (1), $\chi(G[U])\leq 2$.  
        Coloring $G[W]$ with colors $1,2,3,4$ and $G[U]$ with colors $5,6$ gives a $6$-coloring of $G$, regardless of the edges between $W$ and $U$.

		It remains to show that the bound in Theorem~\ref{thm:main} is tight. 
		Consider the 27-vertex 9-regular graph \(G\) depicted in~\autoref{fig:house19273}. 
		Owing to its high symmetry, the verification that \(G\) is $\{P_6, \text{dart}, K_4\}$-free reduces to a finite number of cases and is easily checked. 
		The graph is recorded in the House of Graphs database\footnote{\url{https://houseofgraphs.org/graphs/19273}} as a graph on 27 vertices with \(\omega(G)=3\) and \(\chi(G)=6\).  
		This example confirms that the upper bound in Theorem~\ref{thm:main} cannot be improved.
		
		In conclusion, we complete the proof of Theorem~\ref{thm:main}.
	\end{proof}
	
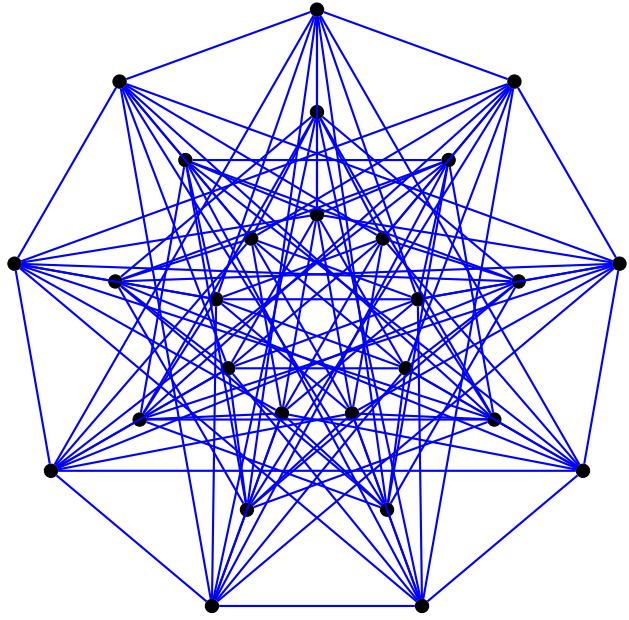
\begin{figure}[htbp]
	\centering
	\tikzstyle{v}=[circle, draw, solid, fill=black, inner sep=0pt, minimum width=5pt]
	\begin{tikzpicture}[scale=1.0]
		\node[v] (0) at (0.31537079828623416, 0.11848533500937997) {};
		\node[v] (1) at (-0.8571428571428541, 2.149338558780869) {};
		\node[v] (2) at (-2.029656512571945, 0.11848533500937819) {};
		\node[v] (3) at (-1.7274143833643754, 1.8325856275068177) {};
		\node[v] (4) at (1.7536717215217015, 3.9068840633207014) {};
		\node[v] (5) at (-2.19047619047619, 1.0305390505444965) {};
		\node[v] (6) at (-2.2463282784782983, -3.0213191669548127) {};
		\node[v] (7) at (1.4878844537153215, -0.5584657395811163) {};
		\node[v] (8) at (2.6603981091444133, -1.2354168141716146) {};
		\node[v] (9) at (0.06898075708077211, -1.7490673081032506) {};
		\node[v] (10) at (1.809523809523811, 1.2656416914891166) {};
		\node[v] (11) at (-0.39408105003104277, -0.47681544925168673) {};
		\node[v] (12) at (-2.597685909585894, 2.869734845413758) {};
		\node[v] (13) at (-0.8571428571428559, 3.503240707961862) {};
		\node[v] (14) at (-1.3202046642546685, -0.47681544925168673) {};
		\node[v] (15) at (-4.857142857142858, 1.5007443324337348) {};
		\node[v] (16) at (-0.8571428571428559, 4.857142857142858) {};
		\node[v] (17) at (0.8834001953001831, 2.869734845413758) {};
		\node[v] (18) at (3.1428571428571423, 1.5007443324337348) {};
		\node[v] (19) at (0.47619047619047805, 1.0305390505444942) {};
		\node[v] (20) at (-3.523809523809522, 1.2656416914891166) {};
		\node[v] (21) at (-3.467957435807414, 3.9068840633207014) {};
		\node[v] (22) at (-4.374683823430123, -1.2354168141716146) {};
		\node[v] (23) at (-3.202170168001033, -0.5584657395811199) {};
		\node[v] (24) at (0.013128669078662725, 1.8325856275068162) {};
		\node[v] (25) at (-1.7832664713664839, -1.7490673081032506) {};
		\node[v] (26) at (0.5320425641925874, -3.0213191669548074) {};
		
		\path[draw, blue, thick]
		(0) edge node {} (1) 
		(0) edge node {} (2) 
		(0) edge node {} (3) 
		(0) edge node {} (4) 
		(0) edge node {} (5) 
		(0) edge node {} (6) 
		(0) edge node {} (7) 
		(0) edge node {} (8) 
		(0) edge node {} (9) 
		(0) edge node {} (10) 
		(1) edge node {} (2) 
		(1) edge node {} (11) 
		(1) edge node {} (12) 
		(1) edge node {} (13) 
		(1) edge node {} (14) 
		(1) edge node {} (15) 
		(1) edge node {} (16) 
		(1) edge node {} (17) 
		(1) edge node {} (18) 
		(2) edge node {} (19) 
		(2) edge node {} (20) 
		(2) edge node {} (21) 
		(2) edge node {} (22) 
		(2) edge node {} (23) 
		(2) edge node {} (24) 
		(2) edge node {} (25) 
		(2) edge node {} (26) 
		(3) edge node {} (4) 
		(3) edge node {} (11) 
		(3) edge node {} (12) 
		(3) edge node {} (13) 
		(3) edge node {} (14) 
		(3) edge node {} (19) 
		(3) edge node {} (20) 
		(3) edge node {} (21) 
		(3) edge node {} (22) 
		(4) edge node {} (15) 
		(4) edge node {} (16) 
		(4) edge node {} (17) 
		(4) edge node {} (18) 
		(4) edge node {} (23) 
		(4) edge node {} (24) 
		(4) edge node {} (25) 
		(4) edge node {} (26) 
		(5) edge node {} (6) 
		(5) edge node {} (11) 
		(5) edge node {} (12) 
		(5) edge node {} (15) 
		(5) edge node {} (16) 
		(5) edge node {} (19) 
		(5) edge node {} (20) 
		(5) edge node {} (23) 
		(5) edge node {} (24) 
		(6) edge node {} (13) 
		(6) edge node {} (14) 
		(6) edge node {} (17) 
		(6) edge node {} (18) 
		(6) edge node {} (21) 
		(6) edge node {} (22) 
		(6) edge node {} (25) 
		(6) edge node {} (26) 
		(7) edge node {} (8) 
		(7) edge node {} (11) 
		(7) edge node {} (13) 
		(7) edge node {} (15) 
		(7) edge node {} (17) 
		(7) edge node {} (19) 
		(7) edge node {} (21) 
		(7) edge node {} (23) 
		(7) edge node {} (25) 
		(8) edge node {} (12) 
		(8) edge node {} (14) 
		(8) edge node {} (16) 
		(8) edge node {} (18) 
		(8) edge node {} (20) 
		(8) edge node {} (22) 
		(8) edge node {} (24) 
		(8) edge node {} (26) 
		(9) edge node {} (10) 
		(9) edge node {} (11) 
		(9) edge node {} (14) 
		(9) edge node {} (16) 
		(9) edge node {} (17) 
		(9) edge node {} (20) 
		(9) edge node {} (21) 
		(9) edge node {} (23) 
		(9) edge node {} (26) 
		(10) edge node {} (12) 
		(10) edge node {} (13) 
		(10) edge node {} (15) 
		(10) edge node {} (18) 
		(10) edge node {} (19) 
		(10) edge node {} (22) 
		(10) edge node {} (24) 
		(10) edge node {} (25) 
		(11) edge node {} (18) 
		(11) edge node {} (22) 
		(11) edge node {} (24) 
		(11) edge node {} (25) 
		(11) edge node {} (26) 
		(12) edge node {} (17) 
		(12) edge node {} (21) 
		(12) edge node {} (23) 
		(12) edge node {} (25) 
		(12) edge node {} (26) 
		(13) edge node {} (16) 
		(13) edge node {} (20) 
		(13) edge node {} (23) 
		(13) edge node {} (24) 
		(13) edge node {} (26) 
		(14) edge node {} (15) 
		(14) edge node {} (19) 
		(14) edge node {} (23) 
		(14) edge node {} (24) 
		(14) edge node {} (25) 
		(15) edge node {} (20) 
		(15) edge node {} (21) 
		(15) edge node {} (22) 
		(15) edge node {} (26) 
		(16) edge node {} (19) 
		(16) edge node {} (21) 
		(16) edge node {} (22) 
		(16) edge node {} (25) 
		(17) edge node {} (19) 
		(17) edge node {} (20) 
		(17) edge node {} (22) 
		(17) edge node {} (24) 
		(18) edge node {} (19) 
		(18) edge node {} (20) 
		(18) edge node {} (21) 
		(18) edge node {} (23) 
		(19) edge node {} (26) 
		(20) edge node {} (25) 
		(21) edge node {} (24) 
		(22) edge node {} (23);
	\end{tikzpicture}
	\caption{The 27-vertex 9-regular graph from House of Graphs \#19273.}
	\label{fig:house19273}
\end{figure}

\vspace{6mm}

\n{\bf Acknowledgements:} 
The research is partially supported by the Foundation for Cultivated Young Talents of Fujian Province, China (Grant No. 2026350294), by the Natural Science Foundation of Fujian Province, China (Grant No. 2026J001968), by the Youth Foundation of Fujian Province (Grant No. JZ240035), by the Minnan Normal University Foundation (Grant No. KJ2023002), and by the Fujian Key Laboratory of Granular Computing and Applications (Minnan Normal University), the Institute of Meteorological Big Data-Digital Fujian, and the Fujian Key Laboratory of Data Science and Statistics.


%

\subsection*{Data availability}
Data sharing is not applicable to this article as no datasets were generated or analysed during the current study.

\subsection*{Conflict of interest}
The authors declare that they have no known competing financial interests or personal relationships that could have appeared to influence the work reported in this paper.

\end{document}